\documentclass[11pt,american,english]{article}
\usepackage{lmodern}
\usepackage{lmodern}
\usepackage[T1]{fontenc}
\usepackage[latin9]{inputenc}
\usepackage[a4paper]{geometry}
\usepackage{babel}
\usepackage{amsmath}
\usepackage{amsthm}
\usepackage{amssymb}
\usepackage{setspace}
\usepackage[unicode=true,pdfusetitle,
 bookmarks=true,bookmarksnumbered=true,bookmarksopen=true,bookmarksopenlevel=3,
 breaklinks=false,pdfborder={0 0 1},backref=false,colorlinks=false]
 {hyperref}

\makeatletter
\theoremstyle{plain}
\newtheorem{thm}{\protect\theoremname}
\theoremstyle{plain}
\newtheorem{prop}[thm]{\protect\propositionname}
\theoremstyle{plain}
\newtheorem{lem}[thm]{\protect\lemmaname}
\theoremstyle{remark}
\newtheorem{rem}[thm]{\protect\remarkname}

\usepackage{babel}
\usepackage{multirow}

\usepackage{arydshln}

\newcommand{\1}{\mbox{1\hspace{-1mm}I}}
\numberwithin{equation}{section}

\usepackage{calc}
  \addto\captionsamerican{\renewcommand{\lemmaname}{Lemma}}
  \addto\captionsamerican{\renewcommand{\propositionname}{Proposition}}
  \addto\captionsamerican{\renewcommand{\remarkname}{Remark}}
  \addto\captionsamerican{\renewcommand{\theoremname}{Theorem}}
  \addto\captionsenglish{\renewcommand{\lemmaname}{Lemma}}
  \addto\captionsenglish{\renewcommand{\propositionname}{Proposition}}
  \addto\captionsenglish{\renewcommand{\remarkname}{Remark}}
  \addto\captionsenglish{\renewcommand{\theoremname}{Theorem}}
  \providecommand{\lemmaname}{Lemma}
  \providecommand{\propositionname}{Proposition}
  \providecommand{\remarkname}{Remark}
\providecommand{\theoremname}{Theorem}

\makeatother

\addto\captionsamerican{\renewcommand{\lemmaname}{Lemma}}
\addto\captionsamerican{\renewcommand{\propositionname}{Proposition}}
\addto\captionsamerican{\renewcommand{\remarkname}{Remark}}
\addto\captionsamerican{\renewcommand{\theoremname}{Theorem}}
\addto\captionsenglish{\renewcommand{\lemmaname}{Lemma}}
\addto\captionsenglish{\renewcommand{\propositionname}{Proposition}}
\addto\captionsenglish{\renewcommand{\remarkname}{Remark}}
\addto\captionsenglish{\renewcommand{\theoremname}{Theorem}}
\providecommand{\lemmaname}{Lemma}
\providecommand{\propositionname}{Proposition}
\providecommand{\remarkname}{Remark}
\providecommand{\theoremname}{Theorem}

\begin{document}
\selectlanguage{american}%
\global\long\def\1{\mbox{1\hspace{-1mm}I}}%
 
\title{IDENTIFICATION OF LINEAR DYNAMICAL SYSTEMS AND MACHINE LEARNING\\
 dedicated to Umberto Mosco, for his 80th birthday }
\author{Alain Bensoussan\\
 International Center for Decision and Risk Analysis\\
 Jindal School of Management, University of Texas at Dallas\thanks{also with the School of Data Science, City University Hong Kong. Research
supported by the National Science Foundation under grants DMS-1612880,
DMS-1905449 and grant from the SAR Hong Kong RGC GRF 11303316. } \\
 Fatih Gelir\\
 Department of Mathematics, University of Texas at Dallas\\
 Viswanath Ramakrishna\\
 Department of Mathematics, University of Texas at Dallas\\
 \thanks{Minh-Binh Tran is partially supported by NSF Grant DMS-1854453, SMU
URC Grant 2020, SMU DCII Research Cluster Grant, Dedman College Linking
Fellowship, Alexander von Humboldt Fellowship.}Minh-Binh Tran\\
 Department of Mathematics, Southern Methodist University }
\maketitle
\selectlanguage{english}%

\section{INTRODUCTION }

The topic of identification of dynamic systems, has been at the core
of modern control , following the fundamental works of Kalman. A good
state of the art for linear dynamic systems can be found in the references
\cite{BRO}, \cite{KFA}, see also \cite{ABE} and \cite{DUP}. Realization
Theory has been one of the major outcomes in this domain, with the
possibility of identifying a dynamic system from an input-output relationship.
The recent development of machine learning concepts has rejuvanated
interest for identification. In this paper, we review briefly the
results of realization theory, and develop some methods inspired by
Machine Learning concepts. We have been inspired by papers \cite{GPI},
\cite{SVV} and \cite{Vidal}.

The interaction between system-control theory and signal processing
on the one hand and machine learning and more generally data science
on the other hand has been steadily increasing in recent years. Given
that all these disciplines may be viewed as part of the activity of
solving inverse problems, this interaction is both inevitable and
inexorable. The papers \cite{GPI,SVV,Vidal} provide compelling instances
of this interaction The paper \cite{GPI} argues this interplay persuasively.
Similarly, \cite{Vidal} discusses the naturality and intervention
of stochastic control and Hamilton-Jacobi theory in the Entropy-Stochastic
Gradient Descent in the study of deep neural networks, amongst many
more such examples in deep learning. \cite{SVV} provides a unified
approach for kernels on dynamical systems used in machine learning
inspired by the behavioural framework in system theory.

\section{REALIZATION THEORY}

\subsection{BASIC PROBLEM}

The basic problem is to go from an input-output relationship to a
dynamical system with state observation and partial observation of
the state

\begin{equation}
x_{t+1}=Ax_{t}+Bv_{t}\label{eq:1-1}
\end{equation}

\[
y_{t}=Cx_{t}
\]

The function $v_{t}$ is the input and the function $y_{t}$ is the
out-put. We have $v_{t}\in R^{m}$ ,$t=1,\cdots$ and $y_{t}\in R^{p},\:t=1,\cdots.$
The map $v\rightarrow y$ is the input-output relationship. If this
map can be written as (\ref{eq:1-1}) then we say that the input-output
relationship has an internal state realization, denoted by $(A,B,C)$.
The function $x_{t}\in R^{n},t=1,\cdots$ is the state of the system.
The number $n$ is called the model order. The identification consists
in finding three matrices $A,B,C$ such that (\ref{eq:1-1}) holds
, given the input-output relationship. We can write the observation
$y_{t}$ as 
\begin{equation}
y_{t+1}=CA^{t}x_{1}+\sum_{s=0}^{t}G_{t-s}v_{s+1},\:t\geq0\label{eq:1-2}
\end{equation}

where

\begin{equation}
G_{t}=CA^{t-1}B,\:t\geq1,\:G_{0}=0\label{eq:1-3}
\end{equation}

are the Markov parameters. We set $\mathcal{G}=(G_{0},G_{1},\cdots),$
called the impulse response of the system.

\subsection{MINIMUM REALIZATION THEORY}

The problem solved in classical dynamic systems theory consists in
finding matrices $A,B,C$ which satisfy (\ref{eq:1-3}) for a large
number of $t$. This research topic has raised a huge amount of work.
It supposes to know the impulse response $\mathcal{G}$ of the dynamic
system. Beautiful results have been obtained to characterize impulse
responses for which there exists an internal state realization, and
the issue of uniqueness. The basic tool is the block Hankel matrix

\begin{equation}
\mathcal{H}_{r,r'}(\mathcal{G})=\left[\begin{array}{ccccc}
G_{1} & G_{2} & G_{3} & \cdots & G_{r'}\\
G_{2} & G_{3} & G_{4} & \dots & G_{r'+1}\\
G_{3} & G_{4} & G_{5} & \cdots & G_{r'+2}\\
\vdots & \vdots & \vdots & \ddots & \vdots\\
G_{r} & G_{r+1} & G_{r+2} & \cdots & G_{r+r'-1}
\end{array}\right]\label{eq:1-4}
\end{equation}

There exists an internal state realization if the Block Hankel matrix
can be written as follows

\begin{equation}
\mathcal{H}_{r,r'}(\mathcal{G})=\mathcal{O}_{r}(C,A)\mathcal{C}_{r'}(A,B),\:\forall r,r'\label{eq:1-5}
\end{equation}

with

\begin{equation}
\mathcal{O}_{r}(C,A)=\left[\begin{array}{c}
C\\
CA\\
\vdots\\
CA^{r-1}
\end{array}\right]\label{eq:1-6}
\end{equation}

\begin{equation}
\mathcal{C}_{r'}(A,B)=\left[\begin{array}{cccc}
B & AB & \cdots & A^{r'-1}B\end{array}\right]\label{eq:1-7}
\end{equation}

The matrix $\mathcal{O}_{r}(C,A)$ is the observability matrix and
the matrix $\mathcal{C}_{r'}(A,B)$ is the controllability matrix.The
pair $A,C$ is said observable if the observability matrix has full
rank. The pair $A,B$ is controllable if the controllability matrix
has full rank. If an internal state realization exists , then it is
minimal if the model order is minimal . Kalman proved the important
result \cite{REK}: A realization $(A,B,C)$ is minimal if and only
if the pair $(A,B)$ is controllable and the pair $(A,C)$ is observable.
A minimal realization is unique up to a change of basis of the state
space. Silverman \cite{SIL} proved the following characterization:
An impulse response $\mathcal{G}$ has a realization if and only if
there exist positive integers $r,r'$ and $\rho$ such that

\begin{equation}
\text{rank}\mathcal{H}_{r,r'}(\mathcal{G})=\text{rank }\mathcal{H}_{r+1,r'+j}(\mathcal{G})=\rho\label{eq:1-8}
\end{equation}

for $j=1,2,\cdots.$ The integer $\rho$ is the minimal order of the
system.

In the sequel , we will consider the dynamic system

\begin{equation}
x_{t+1}=Ax_{t},\:t\geq1\label{eq:200}
\end{equation}

\[
x_{1}=x
\]
with observation

\begin{equation}
y_{t}=Cx_{t}\label{eq:2-3}
\end{equation}

with $x_{t}\in R^{n},$$y_{t}\in R^{p}.$ To simplify we have taken
an input $v_{t}=0$, so there is no way we can learn about a potential
matrix $B.$Because there is no input, the only way to stir the system
is to have a non zero initial state $x.$ To simplify further , we
assume that $x$ and the matrix $C$ are known. The number $n$ is
the model order, which is fixed. So the only unknown is the matrix
$A.$

\section{OBSERVATION OF THE STATE}

We assume here that $C=I$, identity, so the state of the system $x_{t}$
is observable, but the $n\times n$ matrix $A$ is unknown and must
be identified.

\subsection{LEAST SQUARE APPROACH}

If we stack

\[
X_{T}=\left[\begin{array}{c}
x_{2}^{*}\\
\vdots\\
x_{T}^{*}
\end{array}\right]\in\mathcal{L}(R^{n};R^{T-1}),\;Z_{T}=\left[\begin{array}{c}
x_{1}^{*}\\
\vdots\\
x_{T-1}^{*}
\end{array}\right]\in\mathcal{L}(R^{n};R^{T-1})
\]
we can write

\[
X_{T}=Z_{T}A^{*}
\]
and $A^{*}$ can be recovered by

\[
A^{*}=(Z_{T}^{*}Z_{T})^{-1}Z_{T}^{*}X_{T}
\]
provided $Z_{T}^{*}Z_{T}\in\mathcal{L}(R^{n};R^{n})$ is invertible.
So

\begin{equation}
A=\sum_{t=1}^{T-1}x_{t+1}x_{t}^{*}(\sum_{t=1}^{T-1}x_{t}x_{t}^{*})^{-1}\label{eq:2-2}
\end{equation}

\subsection{\label{sec:MACHINE-LEARNING-APPROACH}MACHINE LEARNING APPROACH }

The basic idea is to complete the least square function with a penalty
term. We thus define the function

\begin{equation}
J_{\gamma}(A)=\frac{1}{2}\text{tr}(AA^{*})+\dfrac{\gamma}{2}\sum_{t=1}^{T-1}|x_{t+1}-Ax_{t}|^{2}\label{eq:3-1}
\end{equation}
in which the vectors $x_{t}$ are known. The solution that we get
by this approach is different from (\ref{eq:2-2}). However, it coincides
when $\gamma=+\infty.$ 
\begin{prop}
\label{prop-1} The solution of problem (\ref{eq:3-1}) is given by
formula

\begin{equation}
A^{\gamma}=\sum_{t=1}^{T-1}x_{t+1}x_{t}^{*}(\frac{I}{\gamma}+\sum_{t=1}^{T-1}x_{t}x_{t}^{*})^{-1}\label{eq:3-2}
\end{equation}
\end{prop}

\begin{proof}
The function $J_{\gamma}(A)$ is convex quadratic. The result is obtained
easily from computing the gradient of $J_{\gamma}(A).$$\blacksquare$ 
\end{proof}

\subsection{OTHER FORMULATIONS}

We introduce the vector $p_{t+1},t=1,\cdots,T-1$ by the formula

\begin{equation}
x_{t+1}-A_{\gamma}x_{t}=-\dfrac{1}{\gamma}p_{t+1}\label{eq:3-3}
\end{equation}
then a simple calculation shows that

\begin{equation}
A^{\gamma}=-\sum_{t=1}^{T-1}p_{t+1}x_{t}^{*}\label{eq:3-4}
\end{equation}
So $A_{\gamma}$ appears as a linear combination of the vectors $x_{t}.$
So also, combining (\ref{eq:3-3}) and (\ref{eq:3-4}) we obtain

\begin{equation}
\dfrac{p_{t+1}}{\gamma}+\sum_{s=1}^{T-1}x_{t}.x_{s}p_{s+1}=-x_{t+1}\label{eq:3-5}
\end{equation}
which defines uniquely the coefficients $p_{2},\cdots,p_{T}$ entering
in formula (\ref{eq:3-4}).

\subsection{DUAL PROBLEM }

The system (\ref{eq:3-5}) can be interpreted as a necessary and sufficient
condition of optimality for a different problem, called the dual problem.
The decision is a control $q_{1},\cdots,q_{T-1}$ where $q_{t}\in R^{n}$.
We define the payoff

\begin{equation}
K_{\gamma}(q)=\dfrac{1}{2\gamma}\sum_{t=1}^{T-1}|q_{t}|^{2}+\dfrac{1}{2}\sum_{t,s=1}^{T-1}x_{t}.x_{s}q_{s}.q_{t}+\sum_{t=1}^{T-1}x_{t+1}.q_{t}\label{eq:3-6}
\end{equation}
and the optimal $q=(q_{1},\cdots,q_{T-1}$) is the control $(p_{2},\cdots,p_{T})$
solution of the system (\ref{eq:3-5}).

\subsection{GRADIENT DESCENT ALGORITHM }

Consider the payoff $J(A)=J_{\gamma}(A),$ we drop the index $\gamma$
for simplicity. We can compute the gradient $DJ(A)$ which is a matrix

\begin{equation}
DJ(A)=A(I+\gamma\sum_{t=1}^{T-1}x_{t}x_{t}*)-\gamma\sum_{t=1}^{T-1}x_{t+1}x_{t}^{*}\label{eq:3-7}
\end{equation}
The optimal value of $A,$ noted $A^{\gamma}$ satisfies $DJ(A^{\gamma})=0.$
A gradient descent algorithm is defined by the sequence

\begin{equation}
A^{n+1}=A^{n}-\rho DJ(A^{n})\label{eq:3-8}
\end{equation}
where $\rho$ is a positive number to be chosen conveniently. We use

\[
\dfrac{d}{d\theta}J(A^{n}-\rho\theta DJ(A^{n}))=-\rho\text{tr}DJ(A^{n}-\rho\theta DJ(A^{n}))(DJ(A^{n}))^{*}
\]
So

\begin{equation}
J(A^{n+1})-J(A^{n})=-\rho\text{tr}DJ(A^{n})(DJ(A^{n}))^{*}-\rho\int_{0}^{1}\text{tr }\left(DJ(A^{n}-\rho\theta DJ(A^{n}))-DJ(A^{n})\right)(DJ(A^{n}))^{*}d\theta\label{eq:3-9}
\end{equation}

\[
=-\rho\text{tr}DJ(A^{n})(DJ(A^{n}))^{*}+\rho^{2}\int_{0}^{1}\theta\text{tr (}DJ(A^{n})(I+\gamma\sum_{t=1}^{T-1}x_{t}x_{t}*)(DJ(A^{n}))^{*})d\theta=
\]

\begin{equation}
=(-\rho+\dfrac{\rho^{2}}{2})\text{tr}DJ(A^{n})(DJ(A^{n}))^{*}+\dfrac{\rho^{2}}{2}\gamma\text{tr }\left(DJ(A^{n})\sum_{t=1}^{T-1}x_{t}x_{t}*(DJ(A^{n}))^{*}\right)=\label{eq:3-90}
\end{equation}
\[
\leq\rho(-1+\dfrac{\rho}{2}(1+\gamma\sum_{t=1}^{T-1}|x_{t}|^{2}))\text{tr}DJ(A^{n})(DJ(A^{n}))^{*}
\]
We obtain the 
\begin{prop}
\label{prop3-1} Asume that

\begin{equation}
2<\rho<\dfrac{2}{1+\gamma\sum_{t=1}^{T-1}|x_{t}|^{2}}\label{eq:3-10}
\end{equation}
then $A^{n}\rightarrow A^{\gamma}$ given by formula (\ref{eq:3-2})
which satisfies $DJ(A^{\gamma})=0.$ 
\end{prop}

\begin{proof}
From the assumption (\ref{eq:3-10}) , we have $-1+\dfrac{\rho}{2}(1+\gamma\sum_{t=1}^{T-1}|x_{t}|^{2})<0,$
hence the sequence $J(A^{n})$ is decreasing, thus converging since
it is bounded below. From ( \ref{eq:3-1}) it is clear that the sequence
$A^{n}$ is bounded. We first note that $J(A^{n+1})-J(A^{n})\rightarrow0$.
Moreover, we can extract from $A^{n}$ a subsequence, still denoted
$A^{n}$ which converges to some $A.$ From (\ref{eq:3-90}) we can
immediately write

\[
(1-\dfrac{\rho}{2})\text{tr}DJ(A)(DJ(A))^{*}=\dfrac{\rho}{2}\gamma\text{tr }\left(DJ(A)\sum_{t=1}^{T-1}x_{t}x_{t}*(DJ(A))^{*}\right)
\]
From the condition on $\rho$ the left hand side is negative and the
right hand side positive. Necessarily $DJ(A)=0,$ hence $A=A^{\gamma}.$
Since the limit will be the same for any converging subsequence, the
full sequence converges, which completes the proof. $\blacksquare$ 
\end{proof}

\subsection{RECURSIVITY }

We emphasize here the dependence of $A^{\gamma}$ with respect to
$T.$ So we shall write $A^{T}=A^{\gamma}$ and we want to calculate
$A^{T+1}.$ We first introduce

\begin{equation}
B^{T}=(\frac{I}{\gamma}+\sum_{t=1}^{T-1}x_{t}x_{t}^{*})^{-1}\label{eq:3-11}
\end{equation}
then clearly

\begin{equation}
(B^{T+1})^{-1}=(B^{T})^{-1}+x_{T}x_{T}^{*}\label{eq:3-12}
\end{equation}
and we can see that

\begin{equation}
A^{T+1}=A^{T}+(x_{T+1}-A^{T}x_{T})x_{T}^{*}B^{T+1}\label{eq:3-13}
\end{equation}
In this way, we can compute $A^{T}$ recursively.

\subsection{ASYMPTOTIC ANALYSIS}

We can check easily that the matrix $A^{\gamma}$ converges as $\gamma\rightarrow+\infty$
towards the solution of the least square problem (\ref{eq:2-2}).
In fact we can write the asymptotic exapansion

\begin{equation}
A^{\gamma}=\sum_{t=1}^{T-1}x_{t+1}x_{t}^{*}(\sum_{t=1}^{T-1}x_{t}x_{t}^{*})^{-1}\left(I+\right.\label{eq:3-14}
\end{equation}

\[
\left.\sum_{j=1}^{+\infty}\dfrac{(-1)^{j}}{\gamma^{j}}(\sum_{t=1}^{T-1}x_{t}x_{t}^{*})^{-j}\right)
\]
This result requires the invertibility of the matrix $\sum_{t=1}^{T-1}x_{t}x_{t}^{*}.$
If this is not true, we can state a weaker result . Since the observation
$x_{t}$ is not arbitrary, we may assume that there exists a matrix
$\bar{A}$ such that

\begin{equation}
x_{t+1}=\bar{A}x_{t},\:t=1,\cdots T-1\label{eq:3-15}
\end{equation}
We can state the 
\begin{prop}
\label{prop3-2}Assume the existence of matrices $\bar{A}$ such that
(\ref{eq:3-15}) holds. Then the matrix $A^{\gamma}$ converges as
$\gamma\rightarrow+\infty$ towards the matrix $\bar{A}$ of minimum
norm. 
\end{prop}

\begin{proof}
From (\ref{eq:3-1}) we can write

\begin{equation}
\frac{1}{2}\text{tr}(A^{\gamma}(A^{\gamma})^{*})+\dfrac{\gamma}{2}\sum_{t=1}^{T-1}|x_{t+1}-A^{\gamma}x_{t}|^{2}\leq\frac{1}{2}\text{tr}(\bar{A}(\bar{A})^{*})\label{eq:3-16}
\end{equation}
from which it follows immediately that

\[
A^{\gamma}\text{ is bounded },\:\sum_{t=1}^{T-1}|x_{t+1}-A^{\gamma}x_{t}|^{2}\rightarrow0,\:\text{as}\:\gamma\rightarrow+\infty
\]
So , it is clear that any converging subsequence will tend towards
one matrix $\bar{A}$ satisfying (\ref{eq:3-15}). Thanks to (\ref{eq:3-16})
in which the right hand side refers to any matrix $\bar{A}$ satisfying
(\ref{eq:3-15}), it is clear that the limit point is unique and is
the matrix $\bar{A}$ satisfying (\ref{eq:3-15}) of minimum norm.
This completes the proof of the result. $\blacksquare$ 
\end{proof}

\section{PARTIALLY OBSERVABLE SYSTEM}

\subsection{THE MODEL }

We extend the identification problem above to the case of partially
observable systems. So we have

\begin{equation}
x_{t+1}=Ax_{t}\label{eq:4-1}
\end{equation}

\[
x_{1}=x
\]
and

\begin{equation}
y_{t}=Cx_{t}\label{eq:4-2}
\end{equation}
with $C\in\mathcal{L}(R^{n};R^{d}).$ In the model (\ref{eq:4-1}),
(\ref{eq:4-2}) we suppose that we know the matrix $C$ and the initial
condition $x.$ We want to find the unknown matrix $A$. This problem
generalizes the problem considered in the previous sections, which
is recovered when $C=I.$

\subsection{\label{subsec:A-NATURAL-APPROACH}A NATURAL APPROACH }

Let us assume that the rows of $C$ are linearly independent, which
implies

\begin{equation}
CC^{*}\text{is invertible}\label{eq:4-3}
\end{equation}
then the vector $C^{*}(CC^{*})^{-1}y_{t}$ is solution of (\ref{eq:4-2})
and is the solution with minimum norm. So we can naturally consider
that the state $x_{t},t\geq2$ is in fact reasonably estimated by
$C^{*}(CC^{*})^{-1}y_{t}$ and we are back in the situation of fully
observable systems . So we can estimate $A$ by the formula

\begin{equation}
A_{\gamma}=\sum_{t=1}^{T-1}\hat{x}_{t+1}(\hat{x}_{t})^{*}(\frac{I}{\gamma}+\sum_{t=1}^{T-1}\hat{x}_{t}(\hat{x}_{t})^{*})^{-1}\label{eq:4-4}
\end{equation}
with

\begin{equation}
\hat{x}_{1}=x,\;\hat{x}_{t}=C^{*}(CC^{*})^{-1}y_{t},t=2,\cdots T\label{eq:4-6}
\end{equation}
and we can proceed with similar considerations as above

\subsection{MACHINE LEARNING APPROACH}

A machine learning approach in the spirit of section \ref{sec:MACHINE-LEARNING-APPROACH}
would be to look for $A$ and vectors $x_{t},t=2,\cdots T$ to minimize
the functional

\begin{equation}
J(A,x(.))=\frac{1}{2}\text{tr}(AA^{*})+\frac{\gamma}{2}\sum_{t=1}^{T-1}|x_{t+1}-Ax_{t}|^{2}+\dfrac{\mu}{2}\sum_{t=2}^{T}|y_{t}-Cx_{t}|^{2}\label{eq:4-8}
\end{equation}
with $x_{1}=x.$ In this payoff $x_{t},t=2,\cdots T$ are decision
variables, unlike in the above sections. We note the introduction
of the parameter $\mu.$ The case $\mu=+\infty$ corresponds to the
situation of section \ref{subsec:A-NATURAL-APPROACH}. This problem
leads surprisingly to considerable difficulties. The reason is because
the functional $J(A,x(.))$ is not convex in the pair of arguments
$A,x(.).$ It is convenient to make a change of arguments. We replace
$x(.)$ by $v(.),\:v_{1},\cdots,v_{T-1}$ and define the state $x_{t}$
by the relations

\begin{equation}
x_{t+1}-Ax_{t}=v_{t},\:t=1,\cdots,T-1\label{eq:4-80}
\end{equation}
\[
x_{1}=x
\]
So we define

\begin{equation}
J(A,v(.))=\frac{1}{2}\text{tr}(AA^{*})+\frac{\gamma}{2}\sum_{t=1}^{T-1}|v_{t}|^{2}+\dfrac{\mu}{2}\sum_{t=2}^{T}|y_{t}-Cx_{t}|^{2}\label{eq:4-81}
\end{equation}
with $x_{t}$ defined by (\ref{eq:4-80}). Since the values of $y_{t}$
are not arbitrary, we shall assume that there exists $\bar{A}$ such
that , setting

\begin{equation}
\bar{x}_{t+1}=\bar{A}\bar{x}_{t},\:t=1,\cdots,T-1\label{eq:4-82}
\end{equation}

\[
\bar{x}_{1}=x
\]

\[
y_{t}=C\bar{x}_{t}
\]
so we have the inequality

\begin{equation}
\inf_{A,v(.)}J(A,v(.))\leq\frac{1}{2}\text{tr}(\bar{A}\bar{A}^{*})\label{eq:4-83}
\end{equation}
However, this bound is nor really known, since $\bar{A}$ is not known.
A more practical bound will be

\begin{equation}
\inf_{A,v(.)}J(A,v(.))\leq\dfrac{\mu}{2}\sum_{t=2}^{T}|y_{t}|^{2}\label{eq:4-84}
\end{equation}
This bound depends on the parameter $\mu,$and will not be useful
when we let $\mu\rightarrow+\infty.$

To simplify notation , we shall write $Z=(A,v(.))$ . The space of
vectors $Z$ is called $\mathcal{Z}$ and define the norm in $\mathcal{Z}$
by

\begin{equation}
||Z||^{2}=\text{tr}(AA^{*})+\sum_{t=1}^{T-1}|v_{t}|^{2}\label{eq:4-85}
\end{equation}
We shall compute the gradient $DJ(Z)$ . For that , we introduce the
sequences of vectors $p_{t},t=1,\cdots T$ defined by

\begin{equation}
p_{t}=A^{*}p_{t+1}-\mu C^{*}(y_{t}-Cx_{t}),\:t=1,\cdots T-1\label{eq:4-9}
\end{equation}

\[
p_{T}=-\mu C^{*}(y_{T}-Cx_{T})
\]
We have the 
\begin{lem}
\label{lem4-1}The gradient of the function $J(A,v(.))$ is given
by the formulas 
\end{lem}

\begin{equation}
DJ(Z)=\left|\begin{array}{c}
A+\sum_{t=1}^{T-1}p_{t+1}x_{t}^{*}\\
\\
\gamma v_{t}+p_{t+1},\:t=1,\cdots T-1
\end{array}\right.\label{eq:4-101}
\end{equation}
with $x_{t}$ given by (\ref{eq:4-80}) and $p_{t}$ given by (\ref{eq:4-9}). 
\begin{proof}
A simple calculation yields

\begin{equation}
\cfrac{d}{d\theta}J(Z+\theta\tilde{Z})|_{\theta=0}=\text{tr }A(\tilde{A})^{*}+\gamma\sum_{t=1}^{T-1}v_{t}.\tilde{v}_{t}-\mu\sum_{t=2}^{T}(y_{t}-Cx_{t}).C\tilde{x}_{t}\label{eq:4-10}
\end{equation}
with

\[
\tilde{x}_{t+1}=A\tilde{x}_{t}+\tilde{A}x_{t}+\gamma v_{t}+p_{t+1},t=1,\cdots T-1
\]
\[
\tilde{x}_{1}=0
\]
Using (\ref{eq:4-9}) we get easily

\[
\cfrac{d}{d\theta}J(Z+\theta\tilde{Z})|_{\theta=0}=\text{tr }(A+\sum_{t=1}^{T-1}p_{t+1}x_{t}^{*})(\tilde{A})^{*}+
\]

\[
+\sum_{t=1}^{T-1}(\gamma v_{t}+p_{t+1})\tilde{v}_{t}
\]
and the result follows. $\blacksquare$ 
\end{proof}

\subsection{NECESSARY CONDITIONS OF OPTIMALITY}

A minimum point ( or a local minimum point) $\hat{z}=($$\hat{A},\hat{v}_{t},t=1,\cdots T-1)$
will satisfy the equations $DJ(\hat{Z})=0$ . Therefore

\begin{equation}
\hat{A}+\sum_{t=1}^{T-1}\hat{p}_{t+1}(\hat{x}_{t})^{*}=0\label{eq:4-11}
\end{equation}

\[
\gamma\hat{v}_{t}+\hat{p}_{t+1}=0
\]

\begin{equation}
\hat{x}_{t+1}-\hat{A}\hat{x}_{t}+\dfrac{\hat{p}_{t+1}}{\gamma}=0,\:t=1,\cdots T-1,\:\hat{x}_{1}=x\label{eq:4-12}
\end{equation}
\[
\hat{p}_{t}=(\hat{A})^{*}\hat{p}_{t+1}-\mu C^{*}(y_{t}-C\hat{x}_{t}),t=1,\cdots T-1,\:\hat{p}_{T}=-\mu C^{*}(y_{T}-C\hat{x}_{T})
\]
We claim 
\begin{prop}
\label{prop4-1}We assume (\ref{eq:4-82}). The set of miminimum of
the function $J(Z)$ is not empty and thus the set of triple $\hat{A}$,
$\hat{x}_{t},\hat{p}_{t}$ satisfying (\ref{eq:4-11}), (\ref{eq:4-12})
is not empty. 
\end{prop}

\begin{proof}
In view of (\ref{eq:4-82}) , (\ref{eq:4-83}) holds. Therefore minimizing
sequences remain bounded . Since $J(Z)$ is continuous , the result
follows. 
\end{proof}

\subsection{GRADIENT DESCENT ALGORITHM }

We first show that the function $J(Z)$ has a second derivative $D^{2}J(Z)\in\mathcal{L}(\mathcal{Z};\mathcal{Z})$.
Indeed from (\ref{eq:4-101}) we can easily obtain

\begin{equation}
D^{2}J(Z)\tilde{Z}=\left|\begin{array}{c}
\tilde{A}+\sum_{t=1}^{T-1}\tilde{p}_{t+1}x_{t}^{*}+\sum_{t=1}^{T-1}p_{t+1}(\tilde{x}_{t})^{*}\\
\\
\gamma\tilde{v}_{t}+\tilde{p}_{t+1},\,t=1,\cdots,T-1
\end{array}\right.\label{eq:4-13}
\end{equation}
where $\tilde{Z}=(\tilde{A},\tilde{v}(.))$ and

\begin{equation}
\tilde{x}_{t+1}=A\tilde{x}_{t}+\tilde{A}x_{t}+\tilde{v}_{t},\;\tilde{x}_{1}=0,t=1,\cdots,T-1\label{eq:4-14}
\end{equation}
\[
\tilde{p}_{t}=A^{*}\tilde{p}_{t+1}+p_{t+1}(\tilde{A})^{*}+\mu C^{*}C\tilde{x}_{t},t=1,\cdots,T-1
\]
\[
\tilde{p}_{T}=\mu C^{*}C\tilde{x}_{T}
\]
We also state 
\begin{lem}
\label{lem4-2}We have the formula

\begin{equation}
<D^{2}J(Z)\tilde{Z},\tilde{Z}>=\text{tr }(\tilde{A}(\tilde{A})^{*})+2\sum_{t-1}^{T-1}p_{t+1}.\tilde{A}\tilde{x}_{t}+\gamma\sum_{t-1}^{T-1}|\tilde{v}_{t}|^{2}+\mu\sum_{t=2}^{T}|C\tilde{x}_{t}|^{2}\label{eq:4-15}
\end{equation}
\end{lem}

\begin{proof}
From (\ref{eq:4-13}) we get

\[
<D^{2}J(Z)\tilde{Z},\tilde{Z}>=\text{tr }\left(\tilde{A}(\tilde{A})^{*}+\sum_{t=1}^{T-1}\tilde{p}_{t+1}x_{t}^{*}(\tilde{A})^{*}+\sum_{t=1}^{T-1}p_{t+1}(\tilde{x}_{t})^{*}(\tilde{A})^{*}\right)+
\]
\[
+\gamma\sum_{t-1}^{T-1}|\tilde{v}_{t}|^{2}+\sum_{t=1}^{T-1}\tilde{p}_{t+1}.\tilde{v}_{t}
\]
Using the system (\ref{eq:4-14}), we can compute the term $\sum_{t=1}^{T-1}\tilde{p}_{t+1}.\tilde{v}_{t}$
and after some rearrangements we derive formula (\ref{eq:4-15}) where
$\tilde{p}_{t+1}$ is absent. $\blacksquare$ 
\end{proof}
In the sequel we shall use the properties

\begin{equation}
|<D^{2}J(Z)\tilde{Z},\tilde{Z}>|\leq\varphi(||Z||)||\tilde{Z}||^{2}\label{eq:4-16}
\end{equation}

\begin{equation}
||DJ(Z)||\leq\psi(||Z||)\label{eq:4-17}
\end{equation}
where $\varphi(r),$$\psi(r)$ are continuous and monotone increasing
functions. These properties are consequences of formulas (\ref{eq:4-15})
and (\ref{eq:4-101}) and technical calculations, which we do not
detail. Since we are interested in minimizing $J(Z),$ we can from
(\ref{eq:4-81}) and (\ref{eq:4-84}) consider the ball

\begin{equation}
||Z||\leq M=\sqrt{\frac{\mu}{\min(1,\gamma)}\sum_{t=2}^{T}|y_{t}|^{2}}\label{eq:4-18}
\end{equation}
The gradient descent algorithm is defined by

\begin{equation}
Z^{n+1}=Z^{n}-\rho DJ(Z^{n})\label{eq:4-19}
\end{equation}

\[
J(Z^{1})\leq\dfrac{\mu}{2}\sum_{t=2}^{T}|y_{t}|^{2}\Rightarrow||Z^{1}||\leq M
\]
We can state the 
\begin{thm}
\label{theo4-1}We choose

\begin{equation}
\rho<\min(\frac{2}{\varphi(M+\psi(M))},1)\label{eq:4-20}
\end{equation}
then the sequence $J(Z^{n})$ is decreasing , $||Z^{n}||\leq M$ and
$DJ(Z^{n})\rightarrow0$ , as $n\rightarrow+\infty.$ So the limit
points of the sequence $Z^{n}$ are solutions of $DJ(\hat{Z})=0$. 
\end{thm}

\begin{proof}
We use the formulas

\[
J(Z^{n+1})-J(Z^{n})=-\rho\int_{0}^{1}<DJ(Z^{n}-\rho\theta DJ(Z^{n})),DJ(Z^{n})>d\theta=
\]

\begin{equation}
=-\rho||DJ(Z^{n})||^{2}+\rho^{2}\int_{0}^{1}\int_{0}^{1}\theta<D^{2}J(Z^{n}-\rho\theta\lambda DJ(Z^{n}))DJ(Z^{n}),DJ(Z^{n})>d\lambda d\theta\label{eq:4-21}
\end{equation}
So

\[
J(Z^{n})-J(Z^{n+1})=\rho||DJ(Z^{n})||^{2}-\rho^{2}\int_{0}^{1}\int_{0}^{1}\theta<D^{2}J(Z^{n}-\rho\theta\lambda DJ(Z^{n}))DJ(Z^{n}),DJ(Z^{n})>d\lambda d\theta
\]
Suppose$J(Z^{n})\leq\dfrac{\mu}{2}\sum_{t=2}^{T}|y_{t}|^{2}\Rightarrow$
$||Z^{n}||<M,$ then , from (\ref{eq:4-17}) we have $||DJ(Z^{n})||\leq\psi(M)$
and

\[
||Z^{n}-\rho\theta\lambda DJ(Z^{n})||\leq M+\rho\psi(M)\leq M+\psi(M)
\]
Therefore, from (\ref{eq:4-16}) we get

\[
|<D^{2}J(Z^{n}-\rho\theta\lambda DJ(Z^{n}))DJ(Z^{n}),DJ(Z^{n})>|\leq\varphi(M+\psi(M))||DJ(Z^{n})||^{2}
\]
So

\begin{equation}
J(Z^{n})-J(Z^{n+1})\geq(\rho-\frac{\rho^{2}}{2}\varphi(M+\psi(M)))||DJ(Z^{n})||^{2}\label{eq:4-22}
\end{equation}
Choosing $\rho$ as in (\ref{eq:4-20}) the number $(\rho-\frac{\rho^{2}}{2}\varphi(M+\psi(M)))>0.$
Therefore $J(Z^{n+1})<J(Z^{n})<\dfrac{\mu}{2}\sum_{t=2}^{T}|y_{t}|^{2}\Rightarrow||Z^{n+1}||\leq M.$
We can iterate, and conclude that the sequence $J(Z^{n})$ is monotone
decreasing. It follows that $J(Z^{n})<\dfrac{\mu}{2}\sum_{t=2}^{T}|y_{t}|^{2},\forall n$
and $||Z^{n}||<M,\forall n.$ Looking at the inequality (\ref{eq:4-22})
we get , from the convergence of the sequence $J(Z^{n}),$ that $J(Z^{n})-J(Z^{n+1})\rightarrow0$
hence $||DJ(Z^{n})||\rightarrow0.$ From the continuity of the gradient,
the statement of the Theorem follows. $\blacksquare$ 
\end{proof}
We can detail the steepest gradient. Namely

\begin{equation}
A^{n+1}=A^{n}-\rho(A^{n}++\sum_{t=1}^{T-1}p_{t+1}^{n}(x_{t}^{n})^{*})\label{eq:4-23}
\end{equation}

\[
v_{t}^{n+1}=v_{t}^{n}-\rho(\gamma v_{t}^{n}+p_{t+1}^{n}),\:t=1,\cdots T-1
\]
with

\begin{equation}
x_{t+1}^{n}=A^{n}x_{t}^{n}+v_{t}^{n},\:t=1,\cdots,T-1\label{eq:4-24}
\end{equation}
\[
x_{1}^{n}=x
\]

\begin{equation}
p_{t}^{n}=(A^{n})^{*}p_{t+1}^{n}-\mu C^{*}(y_{t}-Cx_{t}^{n}),\:t=1,\cdots T-1\label{eq:4-25}
\end{equation}

\[
p_{T}^{n}=-\mu C^{*}(y_{T}-Cx_{T}^{n})
\]

\begin{rem}
\label{rem4-1}The algorithm (\ref{eq:4-23}),(\ref{eq:4-24}),(\ref{eq:4-25})
is the straightforward application of the gradient descent method
to the function $J(Z).$ One of the difficulties is to estimate the
bound (\ref{eq:4-20}). 
\end{rem}

\section{SPECIFIC DESCENT METHOD}

\subsection{METHOD }

We exploit here some specific aspects of our optimization problem.
Turning to (\ref{eq:4-11}), (\ref{eq:4-12}) , we write also

\begin{equation}
\hat{A}(\dfrac{I}{\gamma}+\sum_{t=1}^{T-1}\hat{x}_{t}(\hat{x}_{t})^{*})=\sum_{t=1}^{T-1}\hat{x}_{t+1}(\hat{x}_{t})^{*}\label{eq:5-100}
\end{equation}

\begin{equation}
\hat{x}_{t+1}-\hat{A}\hat{x}_{t}+\dfrac{\hat{p}_{t+1}}{\gamma}=0,\:t=1,\cdots T-1,\:\hat{x}_{1}=x\label{eq:5-101}
\end{equation}
\[
\hat{p}_{t}=(\hat{A})^{*}\hat{p}_{t+1}-\mu C^{*}(y_{t}-C\hat{x}_{t}),t=1,\cdots T-1,\:\hat{p}_{T}=-\mu C^{*}(y_{T}-C\hat{x}_{T})
\]
Considering $\hat{A}$ given in the system (\ref{eq:5-101}) we obtain
a unique pair $\hat{x}_{t},\hat{p}_{t},$ since (\ref{eq:5-101})
is the Euler condition of a standard linear quadratic control problem.
We can formulate it as a problem of calculus of variations

\begin{equation}
\min_{x_{2},\cdots,x_{T}}K_{x}(\hat{A},x_{2},\cdots,x_{T})\label{eq:5-102}
\end{equation}
with

\begin{equation}
K_{x}(\hat{A},x_{2},\cdots,x_{T})=\dfrac{\gamma}{2}\sum_{t=1}^{T-1}|x_{t+1}-\hat{A}x_{t}|^{2}+\dfrac{\mu}{2}\sum_{t=2}^{T}|y_{t}-Cx_{t}|^{2},\;x_{1}=x\label{eq:5-103}
\end{equation}
On the other hand, when $\hat{x}_{t}$ is given , with $\hat{x}_{1}=x$,
then $\hat{A}$ defined by (\ref{eq:5-100}) minimizes the function

\begin{equation}
\min_{A}L(A,\hat{x}_{2},\cdots,\hat{x}_{T})\label{eq:5-104}
\end{equation}
with

\begin{equation}
L(A,\hat{x}_{2},\cdots,\hat{x}_{T})=\dfrac{1}{2}\text{tr }AA^{*}+\dfrac{\gamma}{2}\sum_{t=1}^{T-1}|\hat{x}_{t+1}-A\hat{x}_{t}|^{2},\:\hat{x}_{1}=x\label{eq:5-105}
\end{equation}
So $\hat{A}$ appears as the solution of a fixed point problem. We
exploit this fact in designing the algorithm. We define a sequence
$A^{n}$ as follows . For $A^{n}$ given, we define $x_{t}^{n},t=2,\cdots,T$
by minimizing $K_{x}(A^{n},,x_{2},\cdots,x_{T})$ in $x_{2},\cdots,x_{T}.$
We then define $A^{n+1},$by minimizing a modification of $L(A,x_{2}^{n},\cdots,x_{T}^{n})$,
namely

\begin{equation}
L_{\rho}(A,x_{2}^{n},\cdots,x_{T}^{n})=\dfrac{\rho+1}{2}\text{tr }AA^{*}-\rho\text{tr }A^{n}A^{^{*}}+\dfrac{\gamma}{2}\sum_{t=1}^{T-1}|x_{t+1}^{n}-Ax_{t}^{n}|^{2}\label{eq:5-106}
\end{equation}
The parameter $\rho$ is positive. Finally the sequence $A^{n}$ is
defined by

\begin{equation}
x_{t+1}^{n}-A^{n}x_{t}^{n}+\dfrac{p_{t+1}^{n}}{\gamma}=0,\:t=1,\cdots T-1,\:x_{1}^{n}=x\label{eq:5-107}
\end{equation}

\[
p_{t}^{n}=(A^{n})^{*}p_{t+1}^{n}-\mu C^{*}(y_{t}-Cx_{t}^{n}),\;p_{T}^{n}=-\mu C^{*}(y_{T}-Cx_{T}^{n})
\]
\begin{equation}
A^{n+1}(\dfrac{\rho+1}{\gamma}I+\sum_{t=1}^{T-1}x_{t}^{n}(x_{t}^{n})^{*})=\dfrac{\rho}{\gamma}A^{n}+\sum_{t=1}^{T-1}x_{t+1}^{n}(x_{t}^{n})^{*}\label{eq:5-108}
\end{equation}

\subsection{CONVERGENCE}

We have the following convergence result 
\begin{thm}
\label{theo5-1}Assume $\rho\geq0,$ then the sequence $J(A^{n},x^{n}(.))$
( see (\ref{eq:4-8})) is monotone decreasing. The sequence $A^{n},x^{n}(.),p^{n}(.)$
is bounded , $A^{n+1}-A^{n}\rightarrow0$ and limits of converging
subesquences of $A^{n},x^{n}(.),p^{n}(.)$ are solutions of (\ref{eq:5-100}),
(\ref{eq:5-101}). 
\end{thm}

\begin{proof}
We first compute $K_{x}(A^{n+1},,x_{2}^{n},\cdots,x_{T}^{n})-K_{x}(A^{n+1},,x_{2}^{n+1},\cdots,x_{T}^{n+1})>0$
, since $x_{2}^{n+1},\cdots,x_{T}^{n+1}$ minimizes $K_{x}(A^{n+1},,x_{2},\cdots,x_{T}).$
Since it is a quadratic function, we get easily

\begin{equation}
K_{x}(A^{n+1},,x_{2}^{n},\cdots,x_{T}^{n})-K_{x}(A^{n+1},,x_{2}^{n+1},\cdots,x_{T}^{n+1})=\dfrac{\gamma}{2}\sum_{t=1}^{T-1}|x_{t+1}^{n}-x_{t+1}^{n+1}-A^{n+1}(x_{t}^{n}-x_{t}^{n+1})|^{2}+\label{eq:5-109}
\end{equation}

\[
+\dfrac{\mu}{2}\sum_{t=2}^{T}|C(x_{t}^{n}-x_{t}^{n+1})|^{2}
\]
Similarly

\begin{equation}
L_{\rho}(A^{n},x_{2}^{n},\cdots,x_{T}^{n})-L_{\rho}(A^{n+1},x_{2}^{n},\cdots,x_{T}^{n})=\dfrac{\rho+1}{2}\text{tr}(A^{n+1}-A^{n})(A^{n+1}-A^{n})^{*}+\dfrac{\gamma}{2}\sum_{t=1}^{T-1}|(A^{n+1}-A^{n})x_{t}^{n}|^{2}\label{eq:5-110}
\end{equation}
The relation (\ref{eq:5-110}) yields

\begin{equation}
\dfrac{1}{2}\text{tr }A^{n}(A^{n})^{*}+\dfrac{\gamma}{2}\sum_{t=1}^{T-1}|x_{t+1}^{n}-A^{n}x_{t}^{n}|^{2}=\dfrac{1}{2}\text{tr }A^{n+1}(A^{n+1})^{*}+\dfrac{\gamma}{2}\sum_{t=1}^{T-1}|x_{t+1}^{n}-A^{n+1}x_{t}^{n}|^{2}+\label{eq:5-111}
\end{equation}

\[
+(\rho+\dfrac{1}{2})\text{tr}(A^{n+1}-A^{n})(A^{n+1}-A^{n})^{*}
\]
and (\ref{eq:5-109}) yields

\begin{equation}
\dfrac{\gamma}{2}\sum_{t=1}^{T-1}|x_{t+1}^{n}-A^{n+1}x_{t}^{n}|^{2}+\dfrac{\mu}{2}\sum_{t=2}^{T}|y_{t}-Cx_{t}^{n}|^{2}=\dfrac{\gamma}{2}\sum_{t=1}^{T-1}|x_{t+1}^{n+1}-A^{n+1}x_{t}^{n+1}|^{2}+\dfrac{\mu}{2}\sum_{t=2}^{T}|y_{t}-Cx_{t}^{n+1}|^{2}\label{eq:5-112}
\end{equation}

\[
+\dfrac{\gamma}{2}\sum_{t=1}^{T-1}|x_{t+1}^{n}-x_{t+1}^{n+1}-A^{n+1}(x_{t}^{n}-x_{t}^{n+1})|^{2}+\dfrac{\mu}{2}\sum_{t=2}^{T}|C(x_{t}^{n}-x_{t}^{n+1})|^{2}
\]
Adding (\ref{eq:5-111}) and (\ref{eq:5-112}) we obtain

\begin{equation}
\dfrac{1}{2}\text{tr }A^{n}(A^{n})^{*}+\dfrac{\gamma}{2}\sum_{t=1}^{T-1}|x_{t+1}^{n}-A^{n}x_{t}^{n}|^{2}+\dfrac{\mu}{2}\sum_{t=2}^{T}|y_{t}-Cx_{t}^{n}|^{2}=\dfrac{1}{2}\text{tr }A^{n+1}(A^{n+1})^{*}+\dfrac{\gamma}{2}\sum_{t=1}^{T-1}|x_{t+1}^{n+1}-A^{n+1}x_{t}^{n+1}|^{2}+\label{eq:5-113}
\end{equation}

\[
+\dfrac{\mu}{2}\sum_{t=2}^{T}|y_{t}-Cx_{t}^{n+1}|^{2}+(\rho+\dfrac{1}{2})\text{tr}(A^{n+1}-A^{n})(A^{n+1}-A^{n})^{*}+\dfrac{\gamma}{2}\sum_{t=1}^{T-1}|x_{t+1}^{n}-x_{t+1}^{n+1}-A^{n+1}(x_{t}^{n}-x_{t}^{n+1})|^{2}+
\]

\[
+\dfrac{\mu}{2}\sum_{t=2}^{T}|C(x_{t}^{n}-x_{t}^{n+1})|^{2}
\]
It follows that the sequence $\dfrac{1}{2}\text{tr }A^{n}(A^{n})^{*}+\dfrac{\gamma}{2}\sum_{t=1}^{T-1}|x_{t+1}^{n}-A^{n}x_{t}^{n}|^{2}+\dfrac{\mu}{2}\sum_{t=2}^{T}|y_{t}-Cx_{t}^{n}|^{2}$
is decreasing and thus convergente. From (\ref{eq:5-113}) we get
that $A^{n+1}-A^{n}\rightarrow0.$Clearly the sequences $A^{n}$ and
$x_{t}^{n}$ are bounded. From the second relation (\ref{eq:5-107}),
the sequence $p_{t}^{n}$ is also bounded. If we extract a converging
subsequence, the limit is a soltion of the system (\ref{eq:5-100}),
(\ref{eq:5-101}). This concludes the proof 
\end{proof}

\subsection{DUALITY }

In (\ref{eq:4-12}) we replace $\hat{A}$ by its value coming from
(\ref{eq:4-11}). We obtain

\begin{equation}
\hat{x}_{t+1}+\sum_{s-1}^{T-1}\hat{p}_{s+1}\hat{x}_{s}.\hat{x}_{t}+\dfrac{\hat{p}_{t+1}}{\gamma}=0,\:t=1,\cdots T-1,\:\hat{x}_{1}=x\label{eq:5-114}
\end{equation}
\[
\hat{p}_{t}=-\sum_{s=1}^{T-1}\hat{x}_{s}\hat{p}_{s+1}.\hat{p}_{t+1}-\mu C^{*}(y_{t}-C\hat{x}_{t}),t=1,\cdots T-1,\:\hat{p}_{T}=-\mu C^{*}(y_{T}-C\hat{x}_{T})
\]
The unknowns are the pair $\hat{x}_{t},\hat{p}_{t},$$t=1,\cdots T.$
The first one is linear in $\hat{p}(.)$ and the second one is linear
in $\hat{x}(.).$ We can interpret the first equation as the Euler
equation for the the optimization of the functional

\begin{equation}
K(q(.))=\dfrac{1}{2\gamma}\sum_{t=1}^{T-1}|q_{t}|^{2}+\dfrac{1}{2}\sum_{t,s=1}^{T-1}\hat{x}_{t}.\hat{x}_{s}q_{s}.q_{t}+\sum_{t=1}^{T-1}\hat{x}_{t+1}.q_{t}\label{eq:5-115}
\end{equation}
and $\hat{p}_{t+1},t=1,\cdots,T-1$ attains the minimal value of $K(q).$
Unfortunately, this observation is not very useful, since we do not
know the vectors $\hat{x}_{t}$. One can think, of course, of using
the linear system , described by the second equation (\ref{eq:5-114})
to obtain the vectors $\hat{x}_{t}$ , but this system is not immediately
well posed. So , it is not clear how to design an iteration for the
pair of equations (\ref{eq:5-114}). Another possibility to introduce
duality is to consider the dual problem of $K_{x}(\hat{A},x_{2},\cdots,x_{T}).$
It consists in considering $\hat{p}_{t}$ as a state and $\hat{x}_{t}$
as an adjoint state. We can consider indeed the following control
problem.

The evolution of the system is described by the following backward
dynamics: The control is a sequence $z_{2},\cdots z_{T}$ of vectors
in $R^{d},$and we state

\begin{equation}
q_{T}=-\mu C^{*}y_{T}+C^{*}z_{T}\label{eq:5-116}
\end{equation}

\[
q_{t}=(\hat{A})^{*}q_{t+1}-\mu C^{*}y_{t}+C^{*}z_{t},\:t=T-1,\cdots2
\]
\[
q_{1}=(\hat{A})^{*}q_{2}-\mu C^{*}y_{2}+\mu C^{*}Cx
\]
and we minimize the functional

\begin{equation}
\mathcal{K}(z(.))=-q_{1}.x+\dfrac{1}{2\gamma}\sum_{t=2}^{T}|q_{t}|^{2}+\dfrac{1}{2\mu}\sum_{t=2}^{T}|z_{t}|^{2}\label{eq:5-117}
\end{equation}
then the solution is $z_{t}=\mu C\hat{x}_{t}$ and the optimal state
is $\hat{p}_{t}.$ We can then design the following algoritm . Assuming
$A^{n}$ known, we obtain $x_{t}^{n},t=2,\cdots,T$ by minimizing
$K_{x}(A^{n},,x_{2},\cdots,x_{T})$ in $x_{2},\cdots,x_{T}.$ We then
obtain $p_{t}^{n}$ by minimizing the functional $\mathcal{K}(A^{n},z(.))$
defined by the following relations

\begin{equation}
q_{T}=-\mu C^{*}y_{T}+C^{*}z_{T}\label{eq:5-118}
\end{equation}

\[
q_{t}=(A^{n})^{*}q_{t+1}-\mu C^{*}y_{t}+C^{*}z_{t},\:t=T-1,\cdots2
\]
\[
q_{1}=(A^{n})^{*}q_{2}-\mu C^{*}y_{2}+\mu C^{*}Cx
\]
and

\begin{equation}
\mathcal{K}(A^{n},z(.))=-q_{1}.x+\dfrac{1}{2\gamma}\sum_{t=2}^{T}|q_{t}|^{2}+\dfrac{1}{2\mu}\sum_{t=2}^{T}|z_{t}|^{2}\label{eq:5-119}
\end{equation}
Then , we can define $A^{n+1}$ by the formula

\begin{equation}
A^{n+1}=-\sum_{t=1}^{T-1}p_{t+1}^{n}(x_{t}^{n})^{*}\label{eq:5-120}
\end{equation}
This algorithm is different from (\ref{eq:5-108}) (with $\rho=0).$
In fact, it corresponds to

\begin{equation}
A^{n+1}=-\gamma(x_{t+1}^{n}-A^{n}x_{t}^{n})(x_{t}^{n})^{*}\label{eq:5-121}
\end{equation}
We do not claim convergence of this algorithm

\subsection{RECURSIVITY }

We consider now the dependence in $T$. We use the notation

\begin{equation}
A(\dfrac{I}{\gamma}+\sum_{t=1}^{T-1}x_{t}(x_{t})^{*})=\sum_{t=1}^{T-1}x_{t+1}(x_{t})^{*}\label{eq:5-10}
\end{equation}

\begin{equation}
x_{t+1}-Ax_{t}+\dfrac{p_{t+1}}{\gamma}=0,\:t=1,\cdots T-1,\:x_{1}=x\label{eq:5-11}
\end{equation}
\[
p_{t}=(A)^{*}p_{t+1}-\mu C^{*}(y_{t}-Cx_{t}),t=1,\cdots T-1,\:p_{T}=-\mu C^{*}(y_{T}-Cx_{T})
\]
The dependence in $T$ can be emphasaized with the notation $A^{T},x_{t}^{T},p_{t}^{T}.$
To obtain resursive formulas, it is essential to rely on classical
results of control theory, which decouple the forward-backward system
of equations (\ref{eq:5-10}),(\ref{eq:5-11}). In fact, a linear
relation holds

\begin{equation}
x_{t}=r_{t}-\Sigma_{t}p_{t}\label{eq:5-4}
\end{equation}
By well known calculations we have the formulas

\begin{equation}
\Sigma_{t+1}=A\Sigma_{t}A^{*}+\dfrac{I}{\gamma}-A\Sigma_{t}C^{*}(C\Sigma_{t}C^{*}+\dfrac{I}{\mu})^{-1}C\Sigma_{t}A^{*}\label{eq:5-5}
\end{equation}

\[
\Sigma_{1}=0
\]

\begin{equation}
r_{t+1}=Ar_{t}+A\Sigma_{t}C^{*}(C\Sigma_{t}C^{*}+\dfrac{I}{\mu})^{-1}(y_{t}-Cr_{t})\label{eq:5-6}
\end{equation}
\[
r_{1}=x
\]
and then the sequence $p_{t}$ is defined by

\begin{equation}
p_{t}=(I+\mu C^{*}C\Sigma_{t})^{-1}\left(A^{*}p_{t+1}-\mu C^{*}(y_{t}-Cr_{t})\right)\label{eq:5-7}
\end{equation}
\[
p_{T}=-\mu(I+\mu C^{*}C\Sigma_{T})^{-1}C^{*}(y_{T}-Cr_{T})
\]
In the calculations, we have used the fact that $\Sigma_{t}$ is symmetric
and we have the relation

\begin{equation}
(I+\mu C^{*}C\Sigma_{t})^{-1}=I-C^{*}(C\Sigma_{t}C^{*}+\dfrac{I}{\mu})^{-1}C\Sigma_{t}\label{eq:5-8}
\end{equation}
The important point is that $\Sigma_{t},r_{t}$ do not depend on $T.$
Reinstating the notation $T$, we have the formulas

\begin{equation}
p_{t}^{T}=(I+\mu C^{*}C\Sigma_{t})^{-1}\left((A^{T})^{*}p_{t+1}^{T}-\mu C^{*}(y_{t}-Cr_{t})\right),t=1,\cdots T-1\label{eq:5-9}
\end{equation}
\[
p_{T}^{T}=-\mu(I+\mu C^{*}C\Sigma_{T})^{-1}C^{*}(y_{T}-Cr_{T})
\]

\begin{equation}
A^{T}=-\sum_{t=1}^{T-1}p_{t+1}^{T}(r_{t}-\Sigma_{t}p_{t}^{T})^{*}\label{eq:5-12}
\end{equation}
We write

\begin{equation}
A^{T,T+1}=A^{T+1}-A^{T}\label{eq:5-13}
\end{equation}
\[
p_{t}^{T,T+1}=p_{t}^{T+1}-p_{t}^{T},t=1,\cdots T
\]
then , we get the formulas

\begin{equation}
A^{T,T+1}=-p_{T+1}^{T+1}r_{T}^{*}+p_{T+1}^{T+1}(p_{T}^{T+1})^{*}\Sigma_{T}+\label{eq:5-14}
\end{equation}

\[
-\sum_{t=1}^{T-1}p_{t}^{T,T+1}r_{t}^{*}+\sum_{t=1}^{T-1}p_{t+1}^{T,T+1}(p_{t}^{T})^{*}\Sigma_{t}+\sum_{t=1}^{T-1}p_{t+1}^{T}(p_{t}^{T,T+1})^{*}\Sigma_{t}+\sum_{t=1}^{T-1}p_{t+1}^{T,T+1}(p_{t}^{T,T+1})^{*}
\]

\begin{equation}
p_{t}^{T,T+1}=(I+\mu C^{*}C\Sigma_{t})^{-1}\left((A^{T,T+1})^{*}p_{t+1}^{T}+(A^{T})^{*}p_{t+1}^{T,T+1}+(A^{T,T+1})^{*}p_{t+1}^{T,T+1}\right),t=1,\cdots T-1\label{eq:5-15}
\end{equation}

\[
p_{T}^{T,T+1}=(I+\mu C^{*}C\Sigma_{T})^{-1}(A^{T}+A^{T,T+1})^{*}p_{T+1}^{T+1}
\]
We obtain recursivity , but at the price of complex equations.

\subsection{ASYMPTOTIC ANALYSIS}

We take $\mu=\gamma$ and emphasize the dependence in $\gamma$ as
follows:

\begin{equation}
J_{\gamma}(A,x(.))=\frac{1}{2}\text{tr}(AA^{*})+\frac{\gamma}{2}\sum_{t=1}^{T-1}|x_{t+1}-Ax_{t}|^{2}+\dfrac{\gamma}{2}\sum_{t=2}^{T}|y_{t}-Cx_{t}|^{2}\label{eq:5-16}
\end{equation}
and the Euler necessary conditions of optimality

\begin{equation}
A^{\gamma}=-\sum_{t=1}^{T-1}p_{t+1}^{\gamma}(x_{t}^{\gamma})^{*}=0\label{eq:5-17}
\end{equation}

\begin{equation}
x_{t+1}^{\gamma}-A^{\gamma}x_{t}^{\gamma}+\dfrac{p_{t+1}^{\gamma}}{\gamma}=0,\:t=1,\cdots T-1,\:x_{1}^{\gamma}=x\label{eq:5-18}
\end{equation}
\[
p_{t}^{\gamma}=(A^{\gamma})^{*}p_{t+1}^{\gamma}-\gamma C^{*}(y_{t}-Cx_{t}^{\gamma}),t=1,\cdots T-1,\:p_{t}^{\gamma}=-\gamma C^{*}(y_{T}-Cx_{T}^{\gamma})
\]
We want to study the behavior of these quantities as $\gamma\rightarrow+\infty.$
We assume the existence of a matrix $\bar{A}$ such that

\begin{equation}
y_{t}=C\bar{x}_{t}\label{eq:5-19}
\end{equation}

\[
\bar{x}_{t+1}=\bar{A}\bar{x}_{t},\:\bar{x}_{1}=x
\]
We first state the 
\begin{prop}
\label{prop5-1}Assume (\ref{eq:5-19}). Let $A^{\gamma},x^{\gamma}(.)$
be a minimum of $J_{\gamma}(A,x(.))$, then as $\gamma\rightarrow+\infty,$$A^{\gamma}$
converges towards the element $\bar{A}$ satisfying (\ref{eq:5-19})
of minimum norm. 
\end{prop}

\begin{proof}
The proof is similar to that of Proposition \ref{prop3-2}. Necessarily

\[
\frac{1}{2}\text{tr}(A^{\gamma}(A^{\gamma})^{*})+\frac{\gamma}{2}\sum_{t=1}^{T-1}|x_{t+1}^{\gamma}-A^{\gamma}x_{t}^{\gamma}|^{2}+\dfrac{\gamma}{2}\sum_{t=2}^{T}|y_{t}-Cx_{t}^{\gamma}|^{2}\leq\frac{1}{2}\text{tr}(\bar{A}(\bar{A})^{*})
\]
Therefore the sequence $A^{\gamma}$ is bounded. Hence also the sequence
$x_{t}^{\gamma},t=2,\cdots T-1$ is bounded. If we consider a convergingnorm.
subsequence, the limit satisfies necessarily (\ref{eq:5-19}) and
has minimum. 
\end{proof}
We next consider the triple $A^{\gamma},x_{t}^{\gamma},p_{t}^{\gamma},\:t=1,\cdots T$
solution of (\ref{eq:5-17}), (\ref{eq:5-18}). We look for an asympotic
expansion of the form

\begin{equation}
x_{t}^{\gamma}=\bar{x}_{t}+\sum_{j=1}^{+\infty}\dfrac{x_{t}^{j}}{\gamma^{j}}\label{eq:5-20}
\end{equation}

\[
p_{t}^{\gamma}=p_{t}^{0}+\sum_{j=1}^{+\infty}\dfrac{p_{t}^{j}}{\gamma^{j}},\:A^{\gamma}=\bar{A}+\sum_{j=1}^{+\infty}\dfrac{A^{j}}{\gamma^{j}}
\]
After easy but tedious calculations, we obtain the sequence of systems
,$j\geq1$

\begin{equation}
x_{t+1}^{j}-\bar{A}x_{t}^{j}-A^{j}\bar{x}_{t}-\sum_{k=1}^{j-1}A^{k}x_{t}^{j-k}+p_{t+1}^{j-1}=0,\:t=1,\cdots T-1\label{eq:5-21}
\end{equation}
\[
p_{t}^{j-1}=(\bar{A})^{*}p_{t+1}^{j-1}+\sum_{k=1}^{j-1}(A^{k})^{*}p_{t+1}^{j-1-k}+C^{*}Cx_{t}^{j}
\]
\[
x_{1}^{j}=0,\;p_{T}^{j-1}=C^{*}Cx_{T}^{j}
\]
where the sum $\sum_{k=1}^{j-1}$ disappears for $j=1.$ We add the
relations

\begin{equation}
\bar{A}=-\sum_{t=1}^{T-1}p_{t+1}^{0}(\bar{x}_{t})^{*}\label{eq:5-22}
\end{equation}

\[
A^{j}=-\sum_{t=1}^{T-1}p_{t+1}^{j}(\bar{x}_{t})^{*}-\sum_{t=1}^{T-1}\sum_{k=0}^{j-1}p_{t+1}^{k}(x_{t}^{j-k})^{*}
\]
In the system (\ref{eq:5-21}) the unknowns are the pair $x_{t}^{j},p_{t}^{j-1},t=1,\cdots T.$
The matrices $A^{1},\cdots A^{j}$ are known, as well as the vectors
$x_{t}^{j-k},p_{t+1}^{j-1-k},$ for $k=1,\cdots j-1.$ The first equation
(\ref{eq:5-22}) is an equation for $A^{1}$ and the second equation
(\ref{eq:5-22}) is an equation for $A^{j+1}.$ These systems of equations
are linear in the unknowns, although very complicated. If they have
a solution then the expansion (\ref{eq:5-20}) is solution of (\ref{eq:5-17}),
(\ref{eq:5-18}). We shall focus on the first one, which is generic
for the following ones. Namely, we have to solve the system

\begin{equation}
x_{t+1}^{1}-\bar{A}x_{t}^{1}-A^{1}\bar{x}_{t}+p_{t+1}^{0}=0,\:t=1,\cdots T-1\label{eq:5-23}
\end{equation}

\[
p_{t}^{0}=(\bar{A})^{*}p_{t+1}^{0}+C^{*}Cx_{t}^{1}
\]
\[
x_{1}^{1}=0,\cdots,p_{T}^{0}=C^{*}Cx_{T}^{0}
\]
and

\begin{equation}
\bar{A}=-\sum_{t=1}^{T-1}p_{t+1}^{0}(\bar{x}_{t})^{*}\label{eq:5-24}
\end{equation}
As said earlier, in the system (\ref{eq:5-23}), the unknowns are
$x_{t}^{1}$ and $p_{t}^{0},$and $A^{1}$ is a parameter. We define
$A^{1}$ by solving the equation (\ref{eq:5-24}). We first decouple
the system of forward backward equations (\ref{eq:5-23}). We write

\begin{equation}
x_{t}^{1}=r_{t}^{1}-\Sigma_{t}p_{t}^{0}\label{eq:5-25}
\end{equation}
and standard calculations lead to

\begin{equation}
\Sigma_{t+1}=\bar{A}\left(\Sigma_{t}-\Sigma_{t}C^{*}(C\Sigma_{t}C^{*}+I)^{-1}C\Sigma_{t}\right)(\bar{A})^{*}+I\label{eq:5-26}
\end{equation}

\[
\Sigma_{1}=0
\]

\begin{equation}
r_{t+1}^{1}=\bar{A}\left(I-\Sigma_{t}C^{*}(C\Sigma_{t}C^{*}+I)^{-1}C\right)r_{t}^{1}+A_{1}\bar{x}_{t}\label{eq:5-27}
\end{equation}
\[
r_{1}^{1}=0
\]
then using (\ref{eq:5-25}) in the second equation (\ref{eq:5-23})
leads to the following backward recursion for $p_{t}^{0}$

\begin{equation}
p_{t}^{0}=\left(I-C^{*}(C\Sigma_{t}C^{*}+I)^{-1}C\Sigma_{t}\right)(\bar{A})^{*}p_{t+1}^{0}+C^{*}(C\Sigma_{t}C^{*}+I)^{-1}Cr_{t}^{1}\label{eq:5-28}
\end{equation}
\[
p_{T}^{0}=C^{*}(C\Sigma_{t}C^{*}+I)^{-1}Cr_{T}^{1}
\]
To simplify notation we define

\begin{equation}
\Gamma_{t}=\bar{A}\left(I-\Sigma_{t}C^{*}(C\Sigma_{t}C^{*}+I)^{-1}C\right)\label{eq:5-29}
\end{equation}
\[
\Lambda_{t}=C^{*}(C\Sigma_{t}C^{*}+I)^{-1}C
\]
then we get the system

\begin{equation}
r_{t+1}^{1}=\Gamma_{t}r_{t}^{1}+A_{1}\bar{x}_{t}\label{eq:5-30}
\end{equation}

\[
p_{t}^{0}=(\Gamma_{t})^{*}p_{t+1}^{0}+\Lambda_{t}r_{t}^{1}
\]
\[
r_{1}^{1}=0,\:p_{T}^{0}=\Lambda_{T}r_{T}^{1}
\]
If we use the notation

\begin{equation}
\Phi(t,s)=\Gamma_{t}\cdots\Gamma_{s},\:s=1,\cdots t\label{eq:5-31}
\end{equation}

\[
\Phi(t,t+1)=I
\]
then we obtain

\begin{equation}
r_{t+1}^{1}=\sum_{s=1}^{t}\Phi(t,s+1)A_{1}\bar{x}_{s}\label{eq:5-32}
\end{equation}

\begin{equation}
p_{t+1}^{0}=\sum_{s=t}^{T-1}\Phi^{*}(s,t+1)\Lambda_{s+1}r_{s+1}^{1}\label{eq:5-33}
\end{equation}
and we can write the equation for $A_{1}$

\begin{equation}
\bar{A}=-\sum_{t=1}^{T-1}\sum_{\sigma=1}^{T-1}\left(\sum_{s=\max(\sigma,t)}\Phi^{*}(s,t+1)\Lambda_{s+1}\Phi(s,\sigma+1)\right)A_{1}\bar{x}_{\sigma}(x_{t})^{*}\label{eq:5-34}
\end{equation}
We can apply this formula in the scalar case, with the notation $\bar{A}=\bar{a}$,
$A_{1}=a_{1},$$C=c$ and $T=3.$ We get $\Sigma_{1}=0,\:\Sigma_{2}=1,\:\Sigma_{3}=\dfrac{(\bar{a})^{2}+(1+c^{2})}{1+c^{2}}$
. Next $\Gamma_{1}=\bar{a},\;$$\Gamma_{2}=\dfrac{\bar{a}}{1+c^{2}},\:\Gamma_{3}=\dfrac{\bar{a}(1+c^{2})}{(1+c^{2})^{2}+c^{2}(\bar{a})^{2}}$.
We next have $\Lambda_{1}=c^{2},\:\Lambda_{2}=\dfrac{c^{2}}{1+c^{2}},\;\Lambda_{3}=\dfrac{c^{2}(1+c^{2})}{(1+c^{2})^{2}+c^{2}(\bar{a})^{2}}.$
Then equation (\ref{eq:5-34}) becomes

\begin{equation}
\bar{a}=-(\Lambda_{2}+\Lambda_{3}(\Gamma_{2}+\bar{a})^{2})a_{1}x^{2}\label{eq:5-35}
\end{equation}
which gives the value of $a_{1}.$

\section{CONCLUSION }

The concepts and methods of machine learning are most meaningful when
the system is already described by a state representation and the
state has a physical meaning. Otherwise, if the system is decribed
by an input-output linear map, it is probably better to look for the
minimum realization, which can be obtained by the Ho algorithm \cite{HOK}.
For purely deterministic systems as decsribed here, the best is probably
to try to obtain enough observation to be in the case (\ref{eq:4-3}),
and apply methods of full observation. But , in general, there is
a noise which affects the observation and we cannot reduce the problem
to the full observation case. In this situation, the methods described
above are perfectly applicable. It is clear that the penalty terms
play a considerable role, and must be tuned adequately.

\end{document}